\documentclass[a4paper,10pt,reqno]{amsart}
\usepackage{fullpage}
\usepackage[utf8]{inputenc}
\usepackage{hyperref}
\usepackage{amsmath}
\usepackage{amsfonts}
\usepackage{amssymb}
\usepackage{amsthm}
\usepackage{latexsym}
\usepackage{xcolor}
\usepackage{tikz}
\usepackage{mathtools}
\usepackage{faktor}

\providecommand\given{}
\newcommand\SetSymbol[1][]{%
	\nonscript\:#1 :
	\allowbreak
	\nonscript\:
	\mathopen{}}
\DeclarePairedDelimiterX\Set[1]\{\}{%
	\renewcommand\given{\SetSymbol
	}
	#1
}

\DeclarePairedDelimiter{\abs}{\lvert}{\rvert}
\DeclarePairedDelimiterXPP{\norm}[2]{}{\lVert}{\rVert}{_{#2}}{#1}
\DeclarePairedDelimiterXPP{\Altnorm}[3]{}{\lVert}{\rVert}{_{#2}^{#3}}{#1}
\DeclarePairedDelimiterXPP{\fNorm}[2]{}{\lvert}{\rvert}{_{#2}}{#1}

\definecolor{darkgreen}{rgb}{0.0, 0.2, 0.13}
\definecolor{darkolivegreen}{rgb}{0.33, 0.42, 0.18}
\definecolor{chamoisee}{rgb}{0.63, 0.47, 0.35}
\definecolor{cerulean}{rgb}{0.0, 0.48, 0.65}
\definecolor{coolgrey}{rgb}{0.55, 0.57, 0.67}

\DeclareMathOperator{\WF}{WF}

\DeclareMathOperator{\singsupp}{sing\, supp}

\DeclareMathOperator{\Id}{Id}

\newcommand{\fa}{\;\,\forall\,}
\newcommand{\ex}{\;\,\exists\,}

\newcommand{\N}{\mathbb{N}}
\newcommand{\R}{\mathbb{R}}

\newcommand{\E}{\mathcal{E}}
\newcommand{\D}{\mathcal{D}}
\newcommand{\CC}{\mathcal{C}}
\newcommand{\An}{\mathcal{A}}
\newcommand{\G}{\mathcal{G}}

\newcommand{\F}{\mathcal{F}}
\newcommand{\bG}{\mathbf{G}}
\newcommand{\bM}{\mathbf{M}}

\newcommand{\bN}{\mathbf{N}}
\newcommand{\fN}{\mathfrak{N}}

\newcommand{\fM}{\mathfrak{M}}
\newcommand{\bW}{\mathbf{W}}

\newcommand{\alp}{{\lvert\alpha\rvert}}
\newcommand{\bet}{{\lvert\beta\rvert}}
\newcommand{\xit}{{\lvert\xi\rvert}}

\newcommand{\eps}{\varepsilon}

\newcommand{\Beu}[2]{\mathcal{E}^{( #1 )} ( #2 )}
\newcommand{\Rou}[2]{\mathcal{E}^{\{ #1 \}} ( #2 )}
\newcommand{\DC}[2]{\mathcal{E}^{[ #1 ]} ( #2 )}

\newcommand{\vBeu}[3][P]{\mathcal{E}^{ (#2) } ( #3 ; #1 )}
\newcommand{\vRou}[3][P]{\mathcal{E}^{\{ #2 \}} \left( #3; #1 \right)}
\newcommand{\vDC}[3][P]{\mathcal{E}^{[ #2 ]} ( #3 ; #1  )}

\theoremstyle{plain}
\newtheorem{Thm}{Theorem}[section]

\newtheorem{Lem}[Thm]{Lemma}

\theoremstyle{definition}
\newtheorem{Def}[Thm]{Definition}

\theoremstyle{remark}
\newtheorem{Rem}[Thm]{Remark}

\numberwithin{equation}{section}

\title[Regularity of some operators]{Ultradifferentiable regularity properties\\ of a class of hypoelliptic operators}
\author{Stefan F\"urd\"os}
\address{Institute of Mathematics\\ University of Vienna\\
Oskar-Morgenstern-Platz 1\\
1090 Wien, Austria}
\email{stefan.fuerdoes@univie.ac.at}


\begin{document}
	\begin{abstract}
		We investigate the ultradifferentiable regularity of solutions and vectors of
		a class of hypoelliptic differential operators first introduced by H\"ormander.
		Our results extend results of Hashimoto--Morimoto--Matsuzawa and Bolley--Camus--Metivier. 
	\end{abstract}
	\maketitle
	
\section{Introduction}
The question of regularity of vectors of differential operators,
commonly referred to as the problem of iterates, has recently experienced a surge in interest, see e.g.  \cite{MR4509019}, \cite{MR4707824}, \cite{MR4729912}, \cite{MR5014821}, \cite{MR4783651} or \cite{MR4632849}.
The survey \cite{MR1037999}
gives an overview of the older literature and the connection of the problem of iterates with
the problem of hypoellipticity.

We recall that a smooth function $f\in \E(\Omega)$ defined in an open set $\Omega\subseteq\R^n$
is in the Gevrey class $\G^s(\Omega)$ of order $s\geq 1$ (The case $s=1$ corresponds to the class
of real-analytic functions) if for all compact sets $K\subseteq\Omega$ there are constants $C,h>0$
such that
\begin{equation*}
	\sup_{x\in K}\abs*{D^\alpha f(x)}\leq Ch^{\alp}(k!)^s
\end{equation*} 
for all $\alpha\in\N_0^n$.
If $P$ is a differential operator with real-analytic coefficients on $\Omega$ then
a distribution $u\in \D^\prime(\Omega)$ is an $s$-Gevrey vector, $s\geq 1$, if
$P^k u\in L^2_{loc}(\Omega)$ for all $k\in\N_0$ and for all compact $K\subseteq\Omega$ there are
constants $C,h>0$ such that 
\begin{equation*}
	\norm*{P^k u}{L^2(K)}\leq Ch^k ((dk)!)^s
\end{equation*}
for all $k\in\N_0$.
The space of $s$-Gevrey vectors of $P$ is denoted by $\G^s(\Omega;P)$.

If $P$ is an elliptic differential operator $P$ with real-analytic coefficients on an open set $\Omega\subseteq\R^n$ then $\G^s(\Omega;P)=\G^s(\Omega)$, cf.~\cite{MR557524} (for $s=1$
see also \cite{MR152909} and \cite{MR149329}).
According to \cite{MR504629} the converse is also true for $s>1$.
On the other hand it was shown in \cite{MR654409} for hypoelliptic differential operators $P$
of principal type that $\G^1(\Omega;P)=\G^1(\Omega)$ and  there is a loss
of regularity for $u\in\G^s(\Omega,P)$.
Similar results were proven by Derridj, see \cite{MR4707824}, for
 sum of squares of real vector fields of finite type.

This brief note can be viewed as a continuation of the recent paper \cite{MR4432278} of Schindl and the author, where a novel approach to the problem of iterates was introduced. In particular, we proved,
\cite[Theorem 1.1]{MR4432278}, that
for hypoellipitc operators of principal type
an ultradifferentiable Theorem of Iterates hold with respect to
certain ultradifferentiable classes given by weight functions in
the sense of Braun--Meise--Taylor \cite{MR1052587}.
We focus our study in this paper on a class of hypoelliptic differential operators which were originally introduced by H\"ormander, for a definition see the next section.
In \cite{MR632764} Bolley--Camus investigated the Gevrey hypoellipticy and the regularity of Gevrey
vectors of these operators.
Their results on the Gevrey hypoellipticity were improved by Hashimoto--Morimoto--Matsuzawa \cite{MR711439}, whereas Bolley--Camus--Metivier \cite{MR0745974} sharpened the results on the regularity of Gevrey vectors.

We will analyze the hypoellipticity and regularity of vectors
of operators in this class in the general ultradifferentiable setting given by weight matrices \cite{MR3285413} which generalize the setting
of weight sequences, cf.~e.g.~\cite{MR320743}, and the classes given by
weight functions, see cf.~\cite{MR1052587}.
In particular, we will prove the Theorem of Iterates for this operators
for the same family of Braun--Meise--Taylor classes which have been considered in \cite[Theorem 1.1]{MR4432278}.
Note that by \cite[Remark 2.8]{MR5014821} an analogous statement cannot hold for Denjoy-Carleman classes.

\subsection*{Acknowledgments}
This work was funded in whole or in part by the Austrian Science Fund (FWF) 10.55776/PAT1994924.
 
\section{Statement of main results}
We focus in this article on a class of
linear partial differential operators $P(x,D)$ with real-analytic coefficients  in some 
open set $\Omega$, whose symbols $p(x,\xi)$ satisfy the following conditions:
\begin{enumerate}
	\item There are constants $\rho$ and $\delta$ such that $0\leq \delta <\rho\leq 1$
	and for all compact $K\subset \Omega$ there are constants $C,A>0$ and $M>0$ such that 
	\begin{equation*}
		\abs*{\partial_x^\alpha\partial_\xi^\beta p(x,\xi)}\leq CA^{\alp+\bet}\alp !\bet! \abs*{p(x,\xi)}\xit^{-\rho\alp+\delta\bet}
	\end{equation*}
for $x\in K$, $\xi\in\R^n$ with $\xit\geq M$, $\alpha,\beta\in\N_0^n$.
\item There exists some $d^\prime>0$ and for all compact subsets $\Omega\subseteq K$ 
there are constants $C>0$ and $L>0$ such that 
\begin{equation*}
	\xit^{d^\prime}\leq C\abs*{p(x,\xi)}
\end{equation*} 
for $x\in K$, $\xi\in\R^n$ with $\xit\geq L$.
\end{enumerate}
We denote the space of operators satisfying conditions (1) and (2) by $\An^{d,d^\prime}_{\rho,\delta}(\Omega)$.

As noted above we want to discuss the ultradifferentiable hypoellipticity of operators
in $\An^{d,d^\prime}_{\rho,\delta}(\Omega)$.
In the literature there are mainly two closely related families of ultradifferentiable classes:
first, Denjoy-Carleman classes with weight sequences as defining data, and Braun-Meise-Taylor classes given by weight functions.
\begin{Def}
	A weight sequence is a sequence $\bM=(M_k)_k$ of positive real numbers such that $M_0=1$
	and
	\begin{equation*}
		M_k^2\leq M_{k-1}M_{k+1},\qquad k\in\N.
	\end{equation*}
\end{Def}
A function $f\in\E(\Omega)$ is ultradifferentiable of class $\{\bM\}$, respective of class $(\bM)$,
if for all compact sets $K\subseteq\Omega$ there are constants $C,h>0$, respective for all $h>0$ there exists $C>0$, 
such that
\begin{equation*}
	\sup_{x\in K}\abs*{D^\alpha f(x)}\leq Ch^{\alp}M_{\alp},\qquad \alpha\in\N_0^n.
\end{equation*}
The spaces of functions of class $\{\bM\}$ and $(\bM)$  are denoted by $\Rou{\bM}{\Omega}$
and $\Beu{\bM}{\Omega}$, respectively.
\begin{Def}
	A weight function is a continuous, increasing function $\omega:[0,\infty)\rightarrow[0,\infty)$
	such that $\omega_{[0,1]}\equiv 0$ and the following conditions are satisfied:
	\begin{gather*}
	\omega(2t)=O(\omega(t))\qquad t\longrightarrow\infty,\\
    \log t=o(\omega(t))\qquad t\rightarrow\infty,\\
    \varphi_\omega(t):=\omega(e^t)\text{ is convex.}
	\end{gather*}
\end{Def}
If we consider also the conjugate function of $\varphi_\omega$ given by $\varphi^\ast(t)=\sup_{s\geq 0}(st-\varphi(s))$ then a smooth function
$f\in\E(\Omega)$ is of class $\{\omega\}$ if for all compact sets $K\subseteq\Omega$
there are constants $C,h>0$ such that 
\begin{equation}\label{BMTDefEst}
\sup_{x\in K}    \abs{D^\alpha f(x)}\leq C e^{h^{-1}\varphi^\ast_\omega(h\alp)}\qquad \fa\alpha\in\N_0^n.
\end{equation}
On the other hand $f$ is of class $(\omega)$ if for all compact $K\subseteq\Omega$
and every $h>0$ there exists $C>0$ such that the estimate \eqref{BMTDefEst} is satisfied.

In the following we will use the notation
$[\ast]=\{\ast\},(\ast)$ where $\ast=\omega,\bM$.
If $P$ is a differential operator then we say that $P$ is $[\ast]$-hypoelliptic if
$\singsupp_{[\ast]}Pu=\singsupp_{[\ast]} u$ for every $u\in\D^\prime(\Omega)$.
Here $\singsupp_{[\ast]}u$ is the singular support of $u$ with respect to the
class $\DC{\ast}{\Omega}$, which is defined in the obvious way.
\begin{Thm}\label{HypoMainThm}
	Let $P\in\An^{d,d^\prime}_{\rho,\delta}(\Omega)$ and $\theta=(\rho-\delta)^{-1}$.
	Then the following statements hold:
\begin{enumerate}
	\item $P$ is $\{\bM\}$-hypoelliptic for any semiregular weight sequence $\bM$
	such that $\bG^{\theta}\preceq \bM$.
	\item $P$ is $(\bM)$-hypoelliptic for any semiregular weight sequence $\bM$ with
	$\bG^\theta\lhd \bM$.
	\item $P$ is $\{\omega\}$-hypoelliptic for every weight function
	$\omega$ with $\omega(t)=O(t^{\rho-\delta})$, $t\rightarrow\infty$.
	\item $P$ is $(\omega)$-hypoelliptic for every weight function $\omega$ such that
	$\omega(t)=o(t^{\rho-\delta})$, $t\rightarrow\infty$. 
\end{enumerate}
\end{Thm}
If $P$ is a differential operator with analytic coefficients of order $d$ 
and $\omega$ a weight function then a distribution $u\in\D^\prime(\Omega)$ is an element of
$\vRou{\omega}{\Omega}$ if $P^ku\in L^1_{loc}(\Omega)$ for all $k\in\N_0$ and
for every $V\Subset\Omega$ there are $C,h>0$ such that
\begin{equation}\label{BMTVectors}
    \norm*{P^k u}{L^1(V)}\leq C e^{1/h\varphi^\ast_\omega(hdk)}\qquad \fa k\in\N_0.
\end{equation}
On the other hand, $u$ is an element of $\vBeu{\omega}{\Omega}$ if $P^k u\in L^2_{loc}(\Omega)$ for all $k\in\N_0$ and for every $V\Subset \Omega$ and all $h>0$ there 
is $C>0$ such that \eqref{BMTVectors} is satisfied.
\begin{Thm}\label{ThmIteratesOmega}
	Let $\omega$ be a weight function such that
	\begin{equation}\tag{$\Xi$}\label{omega8}
		\ex H\geq 1:\quad \omega\bigl(t^2\bigr)=O(\omega(Ht)),\qquad t\longrightarrow\infty.
	\end{equation}
Then for any $P\in\An_{\rho,\delta}^{d,d^\prime}(\Omega)$ we have that
\begin{align*}
	\vRou{\omega}{\Omega}&=\Rou{\omega}{\Omega},\\
	\vBeu{\omega}{\Omega}&=\Beu{\omega}{\Omega}.
\end{align*} 
\end{Thm}

\section{Ultradifferentiable Hypoellipticity}
If $\bM$ and $\bN$ are two weight sequences then we set
\begin{align*}
	\bM&\leq\bN & &:\Longleftrightarrow & M_k&\leq N_k & \fa k&\in\N_0,\\
	\bM&\preceq\bN & &:\Longleftrightarrow & \ex C,h>0:\quad M_k&\leq Ch^k N_k & \fa k&\in\N_0,\\
    \shortintertext{and}
    \bM&\lhd\bN & &:\Longleftrightarrow & \fa h>0,\,\ex C>0:\quad M_k&\leq Ch^k N_k &
    \fa k&\in\N_0.
\end{align*}
\begin{Def}
	A weight matrix $\fM$ is a family of weight sequences such that for each pair $\bM,\bN\in\fM$
	we have either $\bM\leq\bN$ or $\bN\leq\bM$.
\end{Def}
Let $\Omega\subseteq\R^n$ be an open set and $\fM$ a weight matrix.
 A smooth function $f\in\E(\Omega)$ is ultradifferentiable of class $\{\fM\}$ if
 for each compact set $K\subseteq\Omega$ there are constants $C,h>0$ and some $\bM\in\fM$
 such that
 \begin{equation}\label{ClassDef}
 	\sup_{x\in K}\abs*{D^\alpha f(x)}\leq Ch^{\alp}M_{\alp},\qquad \fa\alpha\in\N_0^n.
 \end{equation}
We say that $f$ is of class $(\fM)$ if for each compact set $K\subseteq\Omega$ 
and every $h>0$ and $\bM\in\fM$ there is a constant $C>0$ such that \eqref{ClassDef} is satisfied.
We write $\DC{\fM}{\Omega}$ for the space of functions of class $[\fM]$. 

If $\fM$ and $\fN$ are two weight matrices then we write $\fM\{\preceq\}\fN$ if
for all $\bM\in\fM$ there is some $\bN\in\fN$ such that $\bM\preceq\bN$ and
$\fM(\preceq)\fN$ if for all $\bN\in\fN$ there is $\bM\in\fM$ such that $\bM\preceq\bN$. Moreover, $\fM\{\lhd)\bN$ if $\bM\lhd \bN$ for all $\bM\in\fM$ and
all $\bN\in\fN$.
Then $\fM[\preceq]\fN$ implies that $\DC{\fM}{\Omega}\subseteq\DC{\fN}{\Omega}$
whereas $\Rou{\fM}{\Omega}\subseteq\Beu{\fN}{\Omega}$ when $\fM\{\lhd)\fN$.
\begin{Def}
Let $\fM$ be a weight matrix.
\begin{enumerate}
	\item We say that $\fM$ is R-semiregular if the following conditions hold:
	\begin{gather}\label{AnalyticIncl}
	\fa\bM\in\fM\quad	\lim_{k\rightarrow\infty}\sqrt[k]{\frac{M_k}{k!}}=\infty\\
	\fa\bM\in\fM \ex\bN\in\fM	\ex Q>0\quad M_{k+1}\leq Q^{k+1}N_k\qquad \fa k\in\N_0.\label{R-Deriv}
	\end{gather}
\item The weight matrix $\fM$ is B-semiregular if \eqref{AnalyticIncl} and 
\begin{equation}\label{B-Deriv}
\fa\bN\in\fM\ex\bM\in\fM\ex Q>0\quad M_{k+1}\leq Q^{k+1}N_k\qquad \fa k\in\N_0
\end{equation}
hold.
\end{enumerate}
\end{Def}
\begin{Rem}\hspace{1ex}
	\begin{enumerate}
		\item Condition \eqref{AnalyticIncl} means that $\An(\Omega)\subsetneq\DC{\fM}{\Omega}$.
		\item Condition \eqref{R-Deriv} (resp.~\eqref{B-Deriv}) implies that
		the Roumieu class $\Rou{\fM}{\Omega}$ (resp.~the Beurling class $\Beu{\fM}{\Omega}$)
		is closed under derivation.
		\item In fact, if $\fM$ is [semiregular] then $\DC{\fM}{\Omega}$ is closed under composition with real-analytic mappings: Let $\Phi:\Omega_1\rightarrow\Omega_2$ be
		a real-analytic mapping then $f\circ\Phi\in\DC{\fM}{\Omega_1}$ for all $f\in\DC{\fM}{\Omega_2}$.
	\end{enumerate}
\end{Rem}

We are going to need the following characterization of $\DC{\fM}{\Omega}$ from \cite{MR4002151}.
\begin{Thm}[{\cite[Proposition 5.1]{MR4002151}}]\label{FourierCharThm}
	Let $u\in\D^\prime(\Omega)$, $\fM$ a weigth matrix and $p_0\in\Omega$.
	Then the following holds:
	\begin{enumerate}
		\item If $\fM$ is R-semiregular then $u$ is of class $\{\fM\}$ near $p_0$ if and only if
		for some neighborhood $V$ of $p_0$ there is a bounded sequence $u_k\in\E^\prime(\Omega)$
		such that $u_k\vert_V=u\vert_V$ and there are a constant $h>0$ and some $\bM\in\fM$
		so that 
		\begin{equation}\label{FourierChar}
			\sup_{\substack{\xi\in\R^n\\ k\in\N_0}}\frac{\xit^k\abs{\hat{u}_k(\xi)}}{h^kM_k}<\infty.
		\end{equation}
	\item If $\fM$ is B-semiregular then $u$ is of class $(\fM)$ near $p_0$ if and only if 
	for some neighborhood $V$ of $p_0$ there is a bounded sequence $u_k\in\E^\prime(\Omega)$
	such that $u_k\vert_V=u\vert_V$ and \eqref{FourierChar} holds for all $h>0$ and every $\bM\in\fM$.
	\end{enumerate}
\end{Thm}
For the definition and theory of the wavefront set $\WF_{[\fM]} u$ associated to ultradifferentiable classes given by weight matrices we refer to \cite{MR4002151} and
\cite{FuerdoesPreprint}.
\begin{Thm}\label{MatrixHypo}
	Let $P\in\An_{\rho,\delta}^{d,d^\prime}(\Omega)$. Then the following holds:
	\begin{enumerate}
		\item $P$ is $\{\fM\}$-hypoelliptic in $\Omega$ for any R-semiregular weight matrix
		$\fM$ such that $\bG^{\theta}\{\preceq\}\fM$.
		\item $P$ is $(\fM)$-hypoelliptic in $\Omega$ for any B-semiregular weight matrix
		$\fM$ such that $\bG^{\theta}\{\lhd)\fM$.
	\end{enumerate}
\end{Thm}
\begin{proof}
	For any $U\Subset\Omega$ there are operators $F,R:\E^\prime(U)\rightarrow \D^\prime(U)$
	such that $PF=\Id+R$ in $\Omega^\prime$ by \cite[Theorem 3.1]{MR711439}.
	Moreover, the kernel of $R$ is in $\G^{\theta}(U)$ whereas the kernel of $F$, which we also
	denote by $F$, satisfies 
	$\singsupp_\theta F\subseteq \Set{(x,x)\given x\in U}\subseteq U\times U$. 
	Now, if $\fM$ is an R-semiregular weight matrix such that $\bG^\theta\{\preceq\}\fM$,
	then $\singsupp_{\{\fM\}}Fu\subseteq \singsupp_{\{\fM\}}u$ and $Ru\in\Rou{\fM}{U}$
	for all $u\in\E^\prime(U)$ by \cite[Theorem 3.14]{FuerdoesPreprint}.
	Thence
	\begin{equation*}
		\singsupp_{\{\fM\}} u=\singsupp_{\{\fM\}} (PF-R)u\subseteq\singsupp_{\{\fM\}}Fu\subseteq
	\singsupp_{\{\fM\}} u
	\end{equation*}
    for all $u\in\E^\prime(U)$, which proves (1).
    The second part is proven completely analogously.
\end{proof}
Theorem \ref{MatrixHypo} immediately implies Theorem \ref{HypoMainThm}(1),(2).
The other parts of Theorem \ref{HypoMainThm} also follow from Theorem \ref{MatrixHypo}
if we also consider the Remark below.
\begin{Rem}\label{BMTRemark1}
    If $\omega$ is a weight function then the weight matrix $\mathfrak{W}$
    associated to $\omega$ consists of the weight sequences $\bW^\lambda$, $\lambda>0$, given by
    \begin{equation*}
        W^{\lambda}_k=e^{\tfrac{1}{\lambda}\varphi^\ast_\omega(\lambda k)}.
    \end{equation*}
    Then $\mathfrak{W}$ has the following properties:
    \begin{gather*}
        \fa \lambda>0 :\quad W_{j+k}^\lambda\leq W_{j}^{2\lambda}W_k^{2\lambda}\qquad
        \fa j,k\in\N_0,\\
        \fa h>0 \ex A\geq 1 \fa\lambda\geq 0 \ex D\geq 1:\quad h^kW^\lambda_k\leq DW^{A\lambda}_k\qquad \fa k\in\N_0. 
    \end{gather*}
    In fact, $\DC{\omega}{\Omega}=\DC{\mathfrak{W}}{\Omega}$ as topological spaces,
    cf.~\cite{MR3285413}.
    Moreover it is well-known, that $\omega(t)=O(t^{1/s})$, $s\geq 1$ is equivalent to
    $\bG^s\{\preceq\}\mathfrak{W}$ and $\bG^s(\preceq)\mathfrak{W}$.
    In the same manner, $\omega(t)=o(t^{1/s})$ is equivalent to
    $\bG^s\{\lhd)\mathfrak{W}$.
\end{Rem}
\section{Ultradifferentiable scales and a general theorem of iterates}
If $P$ is a differential operator with analytic coefficients of order $d$ 
and $\fM$ is a weight matrix satisfying \eqref{AnalyticIncl} then the space $\vRou{\fM}{\Omega}$ consists
of those distributions $u\in\D^\prime(\Omega)$ such that $P^k u\in L^1_{loc}(\Omega)$
and for all $V\Subset \Omega$ there are $\bM\in\fM$ and $C,h>0$ such that
\begin{equation}\label{MatrixVectors}
\norm*{P^k}{L^1(V)}\leq Ch^k M_{dk}\qquad \fa k\in\N_0,
\end{equation}
whereas a distribution $u\in\D^\prime(\Omega)$ is an element of $\vBeu{\fM}{\Omega}$
if $P^ku\in L^1_{loc}(\Omega)$ for all $k\in\N_0$ and for all $V\Subset \Omega$ and
each $h>0$ there is a constant $C>0$ such that \eqref{MatrixVectors} holds.

If $\omega$ is a weight function with $\omega(t)=o(t)$ and $\mathfrak{W}$ is the weight matrix associated to $\omega$ then, in view of Remark \ref{BMTRemark1},
we have that $\vDC{\omega}{\Omega}=\vDC{\mathfrak{W}}{\Omega}$ for any
differential operator $P$ with analytic coefficients.

The rest of this paper is dedicated to the proof of Theorem \ref{ThmIteratesOmega}
which will be in the same spirit as the proof of \cite[Theorem 1.1]{MR4432278}.
For this we need to recall the definition of ultradifferentiable scales
from \cite[Subsection 4A]{MR4432278}. However, in order to be able to use respective adapt the arguments of the proof of \cite[Theoreme 1.2]{MR632764} we have to modify
the definition of ultradifferentiable scales a little bit.
\begin{Def}
	Let $\Lambda$ be a totally ordered set. A mapping $\zeta: \Lambda\times [0,+\infty)\longrightarrow [0,\infty)$ is said to be a generating function if
	for each $\lambda\in\Lambda$ the function $\zeta_\lambda=\zeta(\lambda,\,.\,)$ is continuous
	and $\zeta_\lambda(t)/t$ is increasing on $t\geq 1$
	and the following conditions are satisfied:
	\begin{gather}
		\zeta_\lambda(0)=0; \\
	\text{The sequence	}\N\ni k\mapsto \zeta(k)-\zeta(k-1)\text{ is increasing;}\\
\label{third}	\lim_{t\rightarrow\infty}\left(\zeta_{\lambda}(t)-\log t\right)=\infty.
	\end{gather}
We associate to a generating function a scale $(\bM^\lambda)_\lambda$ by setting
\begin{equation*}
	M_k^\lambda=\exp\left(\zeta(\lambda,k)\right).
\end{equation*}
\end{Def}
The matrix $\fM_\zeta=\fM=\Set{\bM^\lambda\given\lambda\in\Lambda}$ is the weight matrix associated
to the scale $(\bM^\lambda)_\lambda$.
We see that $\fM_\zeta$ satisfies \eqref{AnalyticIncl} due to \eqref{third}.

We are going to need the following conditions:
\begin{align}
   \label{R-mg} \fa \lambda\in\Lambda \ex \lambda^\prime \ex \gamma>0: \quad \zeta_\lambda(2t)&\leq 2\zeta_{\lambda^\prime}(t)+\gamma t\qquad t\geq 1\\
   \label{B-mg} \fa \lambda\in\Lambda \ex \lambda \ex \gamma>0: \quad \zeta_\lambda(2t)&\leq 2\zeta_{\lambda^\prime}(t)+\gamma t\qquad t\geq 1
\end{align}
Note that \eqref{R-mg} (resp.~\eqref{B-mg}) implies that $\fM_{\zeta}$ satisfies \eqref{R-Deriv} (resp.~\eqref{B-Deriv}), cf.~\cite{MR4509817}.
In order that every weight sequence $\bM^\lambda$ is itself semiregular we need to require that
\begin{equation}\label{Single-Deriv}
    \fa \lambda\in\Lambda \ex \gamma>0:\quad \zeta_\lambda(p+1)-\zeta_\lambda(p)\leq \gamma(p+1)\qquad p\in\N_0.
\end{equation}
Furthermore, we also consider the following conditions
\begin{align}
   \label{R-Fitting} \fa \alpha>0 \fa \lambda \in \Lambda\ex\lambda^\prime\in\Lambda \ex \gamma>0:\quad \alpha\zeta_\lambda(t)&\leq \zeta_{\lambda^\prime}(t)+\gamma(t+1)\qquad t\geq 1,\\
    \fa \alpha>0 \fa \lambda^\prime \in \Lambda\ex\lambda\in\Lambda \ex \gamma>0:\quad \alpha\zeta_\lambda(t)&\leq \zeta_{\lambda^\prime}(t)+\gamma(t+1)\qquad t\geq 1.\label{B-fitting}
\end{align}
Finally, for $\sigma\geq 1$ we say that the ultradifferentiable scale
$(\bM^\lambda_\zeta)$ is superordinated to $\bG^s$ if
\begin{equation*}
    \fa \lambda\in\Lambda:\quad \lim_{t\rightarrow}\left(\frac{\zeta_\lambda(t)}{t}-
    \sigma\log t\right)=\infty.
\end{equation*}
Clearly that means that $\{\bG^s\}\{\lhd)\fM_\zeta$.

We are now able to begin with the proof of Theorem \ref{ThmIteratesOmega}.
To this end assume that $P\in\An_{\rho,\delta}^{d,d^\prime}(\Omega)$.
We are making use of the parametrices $A_N$ of the iterated operators
$P^N$, $N\in\N$, constructed in \cite{MR0745974}.
Here we only briefly summarize the results we are going to need
to analyze the regularity of ultradifferentiable vectors of $P$.
The symbols $a_N$ of the operators $A_N$ have asymptotic expansions
$a_N(x,\xi)\sim\sum_{j\geq 1} a_{N,j}(x,\xi)$ with the symbols $a_{N,j}$
being constructed in \cite[II-1.2]{MR0745974}.
In fact, by \cite[Equation (1.10)]{MR0745974} the symbols $a_N$ are of the
form
\begin{equation*}
    a_N(x,\xi):=\sum_{j=0}^\infty \tilde{\chi}_j(\xi/\lambda)a_{N,j}(x,\xi)
\end{equation*}
for $x\in V\Subset\Omega$ and $\xi\in\R^n$. Here $\lambda\geq 1$ 
to be specified later on. 
Moreover $\tilde{\lambda}_j (\eta)\in\CC^\infty(\R^n)$ is a sequence 
such that $\tilde{\chi}_0(\eta)=0$ for $\abs{\eta}\leq L$ and 
$\tilde{\chi}_0(\eta)$ if $\abs{\eta}\geq 2L$ and for $j\geq 1$ we have that
\begin{equation*}
    \begin{cases}
\tilde{\chi}_j(\eta)=0 & \abs{\eta}\leq L\frac{j}{\rho-\delta};\\
\tilde{\chi}_j(\eta)=1 & \abs{\eta}\geq 2L\frac{j}{\rho-\delta};\\
\abs*{D^\alpha \tilde{\chi}_j(\eta)}\leq C^{\alp +1}j^{\alp(1- 1/(\rho-\delta)}
& \alp\leq j
\end{cases}
\end{equation*}
for a constant $C>0$ independent of $j$, cf.~\cite[pp.~57-58]{MR0745974}.

According to \cite[Corollaire 2.7]{MR0745974} there are $\lambda>0$,
$\eps>0$ and $C>0$ such that for all $\xi\in\R^n$, $\xit\geq C$, all $N\in\N$, $\psi\in\CC^\infty_0(\overline{V})$ and all $u\in\D^\prime(\Omega)$ with $P^ku\in L^1_{loc}(\Omega)$ for
$k\in\{1,\dotsc,N\}$ then 
\begin{equation*}
    \widehat{\psi u}(\xi)=\F\bigl(A_N\psi P^Nu\bigr)(\xi)+ 
    \sum_{k=1}^N\F\bigl(S_k\psi P^{k-1}u\bigr)(\xi)+\sum_{k=1}^N\F\bigl(
    A_k[P,\psi]P^{k-1}u\bigr)(\xi)
\end{equation*}
where $S_k$ are error terms (defined in \cite[p.~63]{MR0745974})
satisfying 
\begin{equation*}
    \abs*{\widehat{S_kv}(\xi)}\leq C^{k+1} \abs{\xi}^{-d^\prime k} e^{\eps \abs{\xi}^{\rho-\delta}}\norm{v}{L^1(V)}
\end{equation*}
for $v\in L^k(V)$ and $\abs{\xi}\geq 2\lambda L$.

Now let $x_0\in V\Subset \Omega$
and we fix also $\psi\in\CC^\infty_0(V)$ equal to $1$ near $x_0$.
We choose also a sequence $\varphi_N\in\D(V)$ such that $\varphi_N=1$ near $x_0$ and $\varphi_N\psi=\varphi_N$ and
for all $\beta\in\N^n_0$ there is a constant $C_\beta$ so that 
\begin{equation*}
	\abs*{D^{\alpha+\beta}\varphi_N(x)}\leq C_\beta^{\alp+1} N^{\alp},\qquad x\in V,\; \alp\leq 2d^\prime N
\end{equation*} 
for all $N\geq 1$.
Moreover, we set $\chi_N(x)=\tilde{\chi}_{qN}(x/2\lambda)$ where $q=[2d^\prime/\rho]+2$.
Then we have 
\begin{equation}\label{Fundamental}
	 \widehat{\varphi_N u}= \widehat{\varphi_N \psi u}
	 =
	 \begin{multlined}[t]
	 	\hat{\varphi}_N\ast(1-\chi_N)\widehat{\psi u}+\hat{\varphi}_N\ast \chi_N \F(A_N \psi P^N u)\\+\sum_{k=1}^N \hat{\varphi}_N\ast\chi_N \F\left(S_k \psi P^{k-1}u\right)+\sum_{k=1}^N
	 	\hat{\varphi}_N\ast\chi_N\F\left(A_k[P,\psi]P^{k-1}u\right)
	 \end{multlined}
\end{equation}
for $u\in\D^\prime(\Omega)$ with $P^ku\in L^1_{loc}(\Omega)$ for all $k\in\N_0$.

We need to estimate the right-hand side of \eqref{Fundamental} first
for $u\in\D^\prime(\Omega)$ satisfying the following estimates:
There are a weight sequence $\bM$ and constants such that
\begin{equation*}
    \norm{P^Nu}{L^2(V)}\leq Ch^N M_{dN},\qquad \fa N\in\N_0. 
\end{equation*}

By the proof of \cite[Corollary 2.9]{MR0745974} 
we obtain that there is a constant $C_1>0$ such that for $\xit> C_1$
and $1\leq j\leq k$ we have that 
\begin{equation*}
	\abs*{\hat{\varphi}_k\ast\chi_N \F\left(S_k \psi P^{k-1}u\right)+\sum_{k=1}^N
	\hat{\varphi}_N\ast\chi_N\F\left(A_k[P,\psi]P^{k-1}u\right)}\leq C_1^{N+1}
N^{\tfrac{d^\prime k}{\rho-\delta}} h^k M_{d(k-1)} \xit^{-d^\prime N-d^\prime k+d-1}.
\end{equation*}
On the other hand, observe that
\begin{equation*}
	\sum_{k=1}^N \frac{M_{d(k-1)}}{\xit^{d^\prime k}}
    \leq\sum_{k=1}^N \left(\frac{\Lambda_{dN}}{\xit^{d^\prime/d}}\right)^{dk}
    \leq \sum_{k=1}^\infty 
	\left(\frac{1}{2^d}\right)^j<\infty 
\end{equation*}
if $\xit^{d^\prime/d}\geq 2 \Lambda_{dk}$. 
Thus we have proven the following statement.
\begin{Lem}
  There exists a constant $C_2>0$
    such that 
    \begin{equation*}
	\abs*{\sum_{k=1}^N
		\hat{\varphi}_N\ast\chi_N\F\left(A_k[P,\psi]P^{k-1}u\right)}\leq C_2^{N+1}N^{\tfrac{d^\prime N}{\rho-\delta}} \xit^{-d^\prime N+d-1}
\end{equation*}
for $\xit\geq 2C_2 \Lambda_{dN}^{d/d^\prime}$.
\end{Lem}

Regarding the third term we have by the proof of \cite[Proposition 2.10]{MR0745974}  
that there are constants
$C_3,C_4>0$ such that
\begin{equation*}
	\abs*{\hat{\varphi}_N\ast \chi_N\F\left(S_k\psi P^{k-1} u\right)}\leq C_4^{N+1} N^{\tfrac{d^\prime N}{\rho-\delta}} h^k M_{dk} \xit^{-d^\prime k^N-d^\prime k +n+1}
\end{equation*}
for $\xit\geq C_3$ and $j=1,\dotsc,k$ and $k\in\N$.
As above we thus get, 
\begin{Lem}
    There exists a constant $C_5>0$ such that
    \begin{equation*}
        \abs*{\sum_{k=1}^N\hat{\varphi}^\prime_N\ast \chi_N\F\bigl(S_k\psi P^{k-1}u\bigr)(\xi)}\leq C_5^{N+1}N^{\frac{d^\prime N}{\rho-\delta}}
        \xit\leq C_5^{N+1} (d^\prime N)^{\frac{d^\prime N}{\rho-\delta}} 
        \xit^{-d^\prime N+n+1}
    \end{equation*}
    for $\xit\geq C_5\Lambda_{dN}^{d/d^\prime}$.
\end{Lem}
For the first term, we observe that by \cite[Proposition 2.12]{MR0745974}  there is a constant $C_7>0$ such that 
for all $N\in\N$ and $\xit\geq C_7$ we have that
\begin{equation*}
	\abs*{\hat{\varphi}_N\ast(1-\chi_N)\widehat{\psi u}}\leq C_7^{N+1}
	\left(d^\prime N \right)^{\frac{d^\prime N}{\rho-\delta}}\xit^{-d^\prime N+n+1}\norm{u}{L^1(V)}.
\end{equation*}

Finally, regarding the second term the proof of \cite[Proposition 2.11]{MR0745974}
gives that 
\begin{equation*}
    \abs*{\hat{\varphi}_N\ast\chi_N\F(A_N\psi P^Nu)(\xi)}\leq 
    C_6^{N+1} h^N M_{dN}\xit^{-d^\prime N+n+1}\qquad \xit\geq C_6.
\end{equation*}
Assume now that $\bM=\bM^\lambda_\zeta$ for some generating function satisfying  \eqref{R-mg}. Iterating \eqref{R-mg} shows that for each 
$\lambda\in\Lambda$ and $\alpha> 0$
there are $\lambda^\prime\in\Lambda$ and $\gamma >0$ such that
\begin{equation*}
    \frac{\zeta_\lambda(\alpha t)}{\alpha t}\leq \frac{\zeta_\lambda\bigl(2^\nu t\bigr)}{2^\nu t}\leq 
    \frac{\zeta_\lambda(t)}{t}+\gamma,\qquad t\geq 1
\end{equation*}
where $\nu\in\N$ is the smallest integer such that $\alpha\leq 2^\nu$.
Thus we obtain 
\begin{equation*}
    \abs*{\hat{\varphi}_N\ast\chi_N\F(A_N\psi P^Nu)(\xi)}\leq 
    C_7^{N+1} h^N e^{\lambda dN}\exp\biggl(\frac{d}{d^\prime}\zeta_{\lambda^\prime}(d^\prime N)\biggr)\xit^{-d^\prime N+n+1}\qquad \xit\geq C_7.
\end{equation*}

Summarizing we have proven the following fact:
if $(\bM^\lambda_\zeta)$ is an ultradifferentiable scale superordinate to 
$\bG^\theta$ such that \eqref{R-mg} 
holds and $u\in \vRou{\bM^\lambda}{\Omega}$ (resp.~$u\in\vBeu{\bM^\lambda}{\Omega}$ then there exist $\lambda^\prime\in\Lambda$ and  a constant $c>0$ and some
  $C_0,h_0>0$ (resp.~for each $h_0>0$ there is a constant $C_0$) such that 
\begin{equation*}
    \abs{\widehat{\varphi_N u}(\xi)}\leq C_0 h_0^N
    \exp\left(\frac{d}{d^\prime}\zeta_{\lambda^\prime}(d^\prime N)\right)
    \xit^{-d^\prime N+d+n}
\end{equation*}
for $\xit\geq c\exp(d/(d^\prime)^2\zeta_\lambda(d^\prime N))$ and $N\geq 1$.
Since the sequence
$(\varphi_N u)_N$ is bounded in $\E^\prime(\Omega)$ the Banach-Steinhaus theorem implies that there is a constant $\mu\geq 0$ such that 
that there is a  constant $C_9>0$ such that
\begin{equation*}
    \abs{\widehat{\varphi_Nu}}\leq C_9^{N+1}  \bigl(1+\xit\bigr)^{\mu}.
\end{equation*}
Thence, in fact, there are $C_0,h_0>0$ (resp.~for every $h_0>0$ there 
exists $C_0>0$) such that
\begin{equation*}
    \xit^{d^\prime N-d-n-\mu}\abs*{\widehat{\varphi_N u}(\xi)}
    \leq C_0 h_0^N \exp\left(\frac{d}{d^\prime} \zeta_\lambda(d^\prime N)\right)
\end{equation*}
for $N\in\N$ and $\xi\in\R^n$.

In order to continue for $N\in\N$ we set
\begin{equation*}
	v_N=\varphi_{\widehat{N}}u,\qquad \widehat{N}=\left\lceil \frac{N+d+n+\mu}{d^\prime}\right\rceil 
\end{equation*}
where $\lceil \rho\rceil$ is the smallest integer larger than $\rho$.
In particular note that $N+d+n+\mu\leq d^\prime \widehat{N}\leq N+d+n+\mu+\tilde{d}$, where $\tilde{d}=\lceil d^\prime\rceil$. 
By the above estimates we obtain 
\begin{equation*}
	\begin{split}
	\xit^N\abs*{\hat{v}_N}&\leq C_0h_0^N 
    \exp\left(\frac{d}{d^\prime}\zeta_{\lambda^\prime}(d^\prime N)\right)
    \xit^{N-d^\prime \widehat{N}+d+n+\mu}\\
	&\leq C_0 h_0^N \exp\left(\frac{d}{d^\prime}\zeta_{\lambda^\prime}(N+d+n+\mu+\tilde{d})\right)\xit^{N-N-d-n-\mu+d+n+\mu}\\
\end{split}
\end{equation*}
for $\xit\geq 1$. If \eqref{R-Fitting} holds then we have that
there exists $\tilde{\lambda}\in\Lambda$ such that there are  $C_0,h_0>0$ (resp.~for each $h_0>0$ there is $C_0>0$)
so that
\begin{equation*}
\xit^N\abs*{\hat{v}_N}\leq C_0h_0^N \exp\left(\zeta_{\tilde{\lambda}}
(N+d+n+\mu+\tilde{d})\right)\qquad \fa \xit\geq 1,\;\fa N\in\N,
\end{equation*}
and therefore, if \eqref{Single-Deriv} holds then there are
$C_0,h_0>0$ (resp.~for every $h_0$ there is $C_0>0$) such that 
\begin{equation*}
\xit^N\abs*{\hat{v}_N}\leq C_0h_0^N \exp\left(\zeta_{\tilde{\lambda}}
(N)\right)\qquad \fa\xit\geq 1,\;\fa N\in\N.
\end{equation*}
Thence using Theorem \ref{FourierCharThm} we obtain that $x_0\notin\singsupp_{\bM^{\tilde{\lambda}}} u$ and since
$x_0\in V\Subset\Omega$ have been chosen arbitrarily we have proven the following Theorem.
\begin{Thm}
    Let $P\in\An^{d,d^\prime}_{\rho,\delta}(\Omega)$ and $(\bM^\lambda_\zeta)$  be an ultradifferentiable scale superordinate
    to $\bG^\theta$, $\theta=\tfrac{1}{\rho-\delta}$.
    If the generating function $\zeta$ additionally satisfies 
    \eqref{R-mg}, \eqref{Single-Deriv} and \eqref{R-Fitting} then
    for any $\lambda$ there is some $\tilde{\lambda}$ such that
    \begin{equation*}
        \vDC{\bM^\lambda}{\Omega}\subseteq\DC{\bM^{\tilde{\lambda}}}{\Omega}.
    \end{equation*}
\end{Thm}
In fact, since our arguments are completely local, we can easily see
that also the following statement is true.
\begin{Thm}\label{MatrixVCThm}
    Let $P\in\An^{d,d^\prime}_{\rho,\delta}(\Omega)$ and $(\bM^\lambda_\zeta)$  be an ultradifferentiable scale superordinate
    to $\bG^\theta$, $\theta=\tfrac{1}{\rho-\delta}$.
    \begin{enumerate}
        \item If \eqref{R-mg} and \eqref{R-Fitting} hold
        then $\vRou{\fM_\zeta}{\Omega}=\Rou{\fM_\zeta}{\Omega}$.
        \item If \eqref{B-mg} and \eqref{B-fitting} hold then
        $\vBeu{\fM_\zeta}{\Omega}=\Beu{\fM_\zeta}{\Omega}$.
    \end{enumerate}
\end{Thm}
Theorem \ref{ThmIteratesOmega} follows now from Theorem \ref{MatrixVCThm} and the following Remark.
\begin{Rem}
    Let $\omega$ be a weight function. According to \cite[Section 5A]{MR4432278} 
    $\zeta_\omega(t,\lambda)=\frac{1}{\lambda}\varphi^\ast_\omega(\lambda t)$ is a generating function.
    The convexity of $\varphi^\ast$ implies that
    \begin{equation*}
        \zeta_\omega(2t,\lambda)\leq 2\zeta_\omega(t,2\lambda).
    \end{equation*}
Hence both \eqref{R-mg} and \eqref{B-mg} are satisfied.
Moreover, by \cite[(A.2)]{MR4188685} \eqref{omega8} implies that
\begin{equation*}
    \ex C,H> 0\fa x>0:\quad C\varphi_\omega^\ast\left(\frac{x}{C}\right)\leq \varphi^\ast_\omega\left(\frac{x}{2}\right)+(\log H)x+C.
\end{equation*}
Arguing as in \cite[p.~2172]{MR4188685}
we see that we can conclude that \eqref{omega8} gives that
\begin{equation*}
    \ex A\geq 1\fa\lambda>0\ex \gamma>0:\quad \zeta_\omega(2t,\lambda)\leq \zeta_\omega(t,A\lambda)+\gamma(t+1)\qquad \fa t>0.
\end{equation*}
Since we have that $\varphi^\ast(s)+\varphi^\ast(t)\leq \varphi^\ast(s+t)$, cf.~\cite{MR3285413}, we obtain in fact
\begin{equation*}
    \ex A\geq 1\fa \lambda>0\ex\gamma>0:\qquad 2\zeta_\omega(t,\lambda)\leq
    \zeta_\omega(t,A\lambda)+\gamma(t+1)\qquad t\geq 0
\end{equation*}
Thence, if $\omega$ satisfies \eqref{omega8} then both \eqref{R-Fitting}
and \eqref{B-fitting} hold for the generating function $\zeta_\omega$.

Moreover, by \cite[Remark 2.8(3)]{MR5014821} we obtain that 
if $\omega$ is a weight function such that \eqref{omega8} holds then
$\omega(t)=o(t^\gamma)$ for any $0<\gamma<1$. 
Now according to \cite[Corollary 5.10]{arXiv:2505.17725}
the fact $\omega(t)=o(t^\gamma)$, $0<\gamma<1$, is equivalent to the fact that $\bG^\sigma\lhd \bW^\lambda$ for all $\lambda>0$, where $\sigma=1/\gamma$.
That means that, if \eqref{omega8} holds for $\omega$ then
\begin{equation}\label{LastEstimate1}
    \fa\sigma>1\fa\lambda>0 \qquad \lim_{\N\ni k\rightarrow\infty}\frac{k^\sigma}{\exp\left(\frac{\varphi^\ast_\omega(\lambda k)}{\lambda k}\right)}=0
\end{equation}
We want to replace $k\in\N$ in the limit with $t>0$. To this end observe that
both $t^\sigma$, $\sigma>1$, and the function $t\mapsto 1/t\varphi^\ast_\omega(t)$
are increasing by \cite{MR1052587}. Thence we obtain for $t\in [k,k+1]$, $k\in\N$, that
\begin{equation*}
    \frac{t^\sigma}{\exp\left(\frac{\varphi^\ast_\omega(\lambda t)}{\lambda t}\right)}\leq \frac{(k+1)^\sigma}{\exp\left(\frac{\varphi^\ast_\omega(\lambda k)}{\lambda k}\right)}=\sum_{\ell=1}^\infty \binom{\sigma}{\ell}\frac{t^{\sigma-\ell}}{\exp\left(\frac{\varphi^\ast_\omega(\lambda k)}{\lambda k}\right)}\xrightarrow[]{\quad t\rightarrow\infty\quad}0
\end{equation*}
by \eqref{LastEstimate1}.
Therefore we have proven that if $\omega$ is a weight function satisfying
\eqref{omega8} then the ultradifferentiable scale generated by 
$\zeta_\omega(t,\lambda)=1/\lambda\varphi_\omega^\ast(\lambda t)$ is
superordinate to $\bG^\sigma$ for all $\sigma\geq 1$.
\end{Rem}
\bibliographystyle{plainurl}
\bibliography{Bolley}
\end{document}